\documentclass[12pt]{amsart}

\usepackage{fullpage,url,amssymb,enumerate,colonequals}
\usepackage{mathrsfs}
\usepackage[section]{placeins}
\usepackage{MnSymbol}
\usepackage{extarrows}
\usepackage{lscape}
\usepackage[all,cmtip]{xy}

\usepackage[OT2,T1]{fontenc}
\usepackage{color}

\usepackage[
        colorlinks, citecolor=darkgreen,
        backref,
        pdfauthor={Nuno Freitas},
]{hyperref}

\usepackage{comment}
\usepackage{multirow}

\numberwithin{equation}{section}

\newtheorem{lemma}[equation]{Lemma}
\newtheorem{theorem}[equation]{Theorem}
\newtheorem{proposition}[equation]{Proposition}

\newtheorem{claim*}{Claim}

\theoremstyle{definition}
\newtheorem{remark}[equation]{Remark}

\newtheorem{example}[equation]{Example}

\definecolor{darkgreen}{rgb}{0,0.5,0}
\definecolor{rem}{rgb}{0.8,0,0}
\definecolor{new}{rgb}{0.3,0.1,0.9}
\definecolor{reply}{rgb}{0,0,0.8}
\definecolor{gray}{gray}{0.7}

\renewcommand{\det}{\text{det}}

\newcommand{\F}{\mathbb{F}}

\newcommand{\Q}{\mathbb{Q}}

\newcommand{\Z}{\mathbb{Z}}
\newcommand{\rhobar}{{\overline{\rho}}}

\newcommand{\calO}{\mathcal{O}}

\newcommand{\Qbar}{{\overline{\Q}}}
\newcommand{\Gal}{\text{Gal}}
\newcommand{\Frob}{\text{Frob}}

\newcommand{\GL}{\text{GL}}

\newcommand{\Ebar}{{\overline{E}}}

\renewcommand{\ker}{\text{ker}\:}

\DeclareMathOperator{\Tr}{\text{Tr}}

\newcommand{\oE}{\overline{E}}

\newcommand{\cL}{\mathcal{L}}

\begin{document}

\title {Symplectic criteria for elliptic curves, revisited}

\author{Nuno Freitas}
\address{Instituto de Ciencias Matemáticas (ICMAT),
          Nicolás Cabrera 13-15
         28049 Madrid, Spain}
\email{nuno.freitas@icmat.es}

\author{Alain Kraus}
\address{Sorbonne Universit\'e,
Universit\'e Paris Cit\'e,
CNRS, IMJ-PRG,
F-75005 Paris, France}
\email{alain.kraus@imj-prg.fr}

\author{Ignasi S\'anchez-Rodr\'iguez}
\address{Universitat de Barcelona (UB),
        Gran Via de les Corts Catalanes 585,
        08007 Barcelona, Spain}
\email{ignasi.sanchez@ub.edu}

\thanks{Freitas was partly supported by the PID 2022-136944NB-I00 grant of the MICINN
(Spain)}
\thanks{S\'anchez-Rodr\'iguez was party supported by the FPI grant of the PID PID2019-107297GB-I00 grant of the MICINN (Spain)}

\subjclass[2010]{}

\begin{abstract}
Let $\ell$ and $p \geq 3$ be different primes. Let $E/\Q_\ell$ and $E'/\Q_\ell$ be elliptic curves with isomorphic $p$-torsion. Assume that $E$ has potentially multiplicative reduction.
We classify when all $G_{\Q_\ell}$-isomorphisms $\phi  : E[p] \to E'[p]$ have the same symplectic type and prove two new criteria to determine the type in that case. In particular, when both curves have multiplicative reduction, our results cover the case of unramified $p$-torsion which is not covered by the original criterion due to Kraus and Oesterl\'e.

We also give a variant of a symplectic criterion for the case when both $E$ and~$E'$ have good reduction and provide an algorithm to apply it.

As an application, we determine the symplectic type of all the mod $p \geq 5$ congruences between rational elliptic curves with conductor $\leq 500 000$ that satisfy the hypothesis of either of our criteria at some prime~$\ell$.
\end{abstract}

\maketitle

\section{Introduction}

Let $p \geq 3$ be a prime and $G_\Q :=\Gal(\Qbar / \Q)$. We consider triples $(E,E',p)$ where $E/\Q$ and $E'/\Q$ are elliptic curves with isomorphic $p$-torsion modules $E[p]$ and~$E[p]$ such that
either
\begin{itemize}
 \item[(i)] every $G_\Q$-modules isomorphism $\phi : E[p] \to E'[p]$
admits a multiple $\lambda \cdot \phi$ with $\lambda \in \F_p^\times$
preserving the Weil pairing; or
\item[(ii)] no $G_\Q$-isomorphism $\phi : E[p] \to E'[p]$ preserves the Weil pairing.
\end{itemize}
We say that the {\it symplectic type} of~$(E,E',p)$ is symplectic in case~(i) and anti-symplectic in case~(ii); we also say that $E[p]$ and $E'[p]$ are {\it symplectically isomorphic} in case~(i) and that they are {\it anti-symplectically isomorphic} in case~(ii).

In~\cite{symplectic}, the first two authors studied the problem of deciding the symplectic type of $(E,E',p)$, using a local approach. Indeed, they first determined a list $\cL_{(E,E',p)}$ of primes $\ell \neq p$ for which the curves $E/\Q_\ell$ and $E'/\Q_\ell$ contain enough information to decide the type of $(E,E',p)$. Second,
for each $\ell \in \cL_{(E,E',p)}$, they proved a {\it symplectic criterion} that allows to decide the symplectic type of $(E,E',p)$ using only local information at $\ell$.
In the published version of~\cite{symplectic} it is incorrectly stated that $\cL_{(E,E',p)}$ is a complete list, but this is only true for the primes where both curves have potentially good reduction; as we will explain in Theorem~\ref{thm:summary}, the results in this work make the list $\cL_{(E,E',p)}$ complete.

Let $G_{\Q_\ell} :=\Gal(\Qbar_\ell / \Q_\ell)$.
For elliptic curves $E/\Q_\ell$ and $E'/\Q_\ell$ with $G_{\Q_\ell}$-isomorphic $p$-torsion, we say that {\it a symplectic criterion exists} if and only if all $G_{\Q_\ell}$-isomorphisms $\phi : E[p] \to E'[p]$ have the same symplectic type; this is the case if and only if $E[p]$ admits only symplectic automorphisms (an automorphism of $E[p]$ is {\it symplectic} if and only if it has a scalar multiple preserving the Weil pairing).

In~\cite{symplectic} the authors classified when a symplectic criterion exists in the case of both $E/\Q_\ell$ and $E'/\Q_\ell$ having potentially good reduction and proved symplectic criteria covering all such cases. They also proved one criterion for the mixed case, that is,  when~$E/\Q_\ell$ has potentially multiplicative reduction and $E'/\Q_\ell$ has potentially good reduction. Moreover, the original symplectic criterion proved by Kraus--Oesterl\'e~\cite{KO} covers the case when both curves have multiplicative reduction and ramified $p$-torsion.

In this work, we determine when a symplectic criterion exists in the mixed reduction case (see Theorem~\ref{thm:existence}) and in the case where both curves have potentially multiplicative reduction (see Theorem~\ref{thm:existenceTate}). We also prove the following two new criteria, one in each of these cases.

Let $\ell$ and~$p\geq 3$ be primes satisfying $\ell \equiv 1 \pmod{p}$. We fix a $p$-th root of unity $\zeta_p \in \F_\ell$ and denote also by $\zeta_p$ its lift in $\Z_\ell^\times \subset \Qbar_\ell$. For $k=\F_\ell$ or $\Q_\ell$ and an elliptic curve $E/k$ we recall that~$P,Q \in E[p]$ is a {\it symplectic basis}  if $e_E(P,Q) =\zeta_p$ where $e_E$ is the Weil pairing.

\begin{theorem} \label{thm:multiplicative} Let $\ell$ and~$p\geq 3$ be primes satisfying $\ell \equiv 1 \pmod{p}$.
Let $E_1$ and $E_2$ be elliptic curves over~$\Q_\ell$ with split multiplicative reduction. Let ${\tilde{j}_i} := j_{E_i} \cdot \ell^{-v(j_{E_i})}$. Suppose that
\[ v(j_{E_i}) \equiv 0 \pmod{p} \quad \text{ and } \quad {\tilde{j}_i}^{\frac{\ell-1}{p}} \not\equiv 1 \pmod{\ell} \quad \text{ for } i = 1,2.
\]
Then, $E_1[p]$ and $E_2[p]$ are $G_{\Q_\ell}$-isomorphic and
not simultaneously symplectically and anti-symplectically isomorphic. Furthermore, they are symplectically isomorphic if and only if~$h_1 / h_2$ is a square mod~$p$,
where
$$h_i := \text{Log}_{\zeta_p}(\tilde{j}_i^{\frac{\ell-1}{p}} \pmod{\ell}),$$
%In particular, $E[p]$ and~$E'[p]$ are not both symplectically and anti-symplectically isomorphic.
where $\text{Log}_{\zeta_p}$ is the discrete logarithm of basis $\zeta_p$
in the subgroup of the $p$-th roots of unity in~ $\F_\ell^\times.$

\end{theorem}

Note that if $E[p] \simeq E'[p]$ as $G_{\Q_\ell}$-modules and both curves have potentially multiplicative reduction, then taking a suitable quadratic twist reduces to the case where both curves have split multiplicative reduction (see \cite[Lemma 5]{symplectic}). Therefore, if a criterion exists, after a taking quadratic twist if needed, by Theorem~\ref{thm:existenceTate}, we can either apply the previous theorem or apply the original criterion of Kraus--Oesterl\'e; these two cases correspond with  $E[p]$ being unramified or ramified, respectively.

\begin{theorem}\label{thm:mixedReduciton}
Let $p \geq 3$ and $\ell$ be primes such that $\ell \equiv 1 \pmod{p}$.
Let $E/\Q_\ell$ and $E'/\Q_\ell$ be elliptic curves
such that $E$ has split multiplicative reduction and $E'$ has good reduction.
Suppose the following hold:
\begin{itemize}
 \item[(1)] $v(j_E) \equiv 0 \pmod{p}$ and ${\tilde{j}_E}^{\frac{\ell-1}{p}} \not\equiv 1 \pmod{\ell}$ where ${\tilde{j}_E} := j_E \cdot \ell^{-v(j_E)}$; and
\item[(2)] $\rhobar_{E',p}(\Frob_\ell)$ has order~$p$, where $\Frob_\ell \in G_{\Q_\ell}$ is a Frobenius element.

\end{itemize}
Then, $E[p]$ and $E'[p]$ are $G_{\Q_\ell}$-isomorphic and
not simultaneously symplectically and anti-symplectically isomorphic. Moreover,
\begin{itemize}
 \item[(i)] there is a symplectic basis of $E'[p]$ such that
$\rhobar_{E',p}(\Frob_\ell) = \left(\begin{smallmatrix}
  1 & h' \\
  0 & 1
\end{smallmatrix} \right)$; and
\item[(ii)] $E[p]$ and $E'[p]$ are symplectically isomorphic if and only if $-h / h'$ is a square mod~$p$ where
$h := \text{Log}_{\zeta_p}(\tilde{j}_i^{\frac{\ell-1}{p}} \pmod{\ell})$.
\end{itemize}
\end{theorem}
Observe that, in the mixed reduction case, Proposition~\ref{P:multTwist} below implies that after twisting both curves by some $d \in \Q_\ell$ if needed, we can assume that either Theorem~\ref{thm:mixedReduciton} applies or $p=3$ and
~\cite[Theorem 14]{symplectic} applies.

We will also prove an alternative formulation of~\cite[Theorem~16]{symplectic}, i.e., the symplectic criterion for the case where both curves have good reduction. The advantage of the following version is that it is computationally accessible; we also provide a {\tt Magma} implementation~\cite{git}.

\begin{theorem} \label{thm:goodred}
Let $\ell$ and $p \geq 3$ be different primes.
Let $E$ and $E'$ be elliptic curves over $\Q_\ell$ with good reduction and isomorphic $p$-torsion modules.
Assume that $p \mid \#\rhobar_{E,p}(\Frob_\ell)$.

Then,
there are symplectic basis of $E[p]$ and $E'[p]$ such that
\begin{equation*}
 \rhobar_{E,p}(\Frob_\ell) =
 \begin{pmatrix} \frac{a_\ell}{2}  & h \\
 0 &  \frac{a_\ell}{2} \end{pmatrix}
  \quad \text{ and } \quad
 \rhobar_{E',p}(\Frob_\ell) =
 \begin{pmatrix} \frac{a_\ell}{2}  & h' \\
 0 &  \frac{a_\ell}{2} \end{pmatrix},
 \label{E:MatrixReps}
\end{equation*}
where $a_\ell := a_\ell(E) \equiv a_\ell(E') \pmod{p}$ and $h$, $h'$ both non-zero. Moreover, $E[p]$ and~$E'[p]$ are symplectically isomorphic if and only if $h/h'$ is a square mod~$p$; also, they are not simultaneously symplectically and anti-symplectically isomorphic.

\label{T:mainAbelian}
\end{theorem}

To put everything together, we will show that the results in this work together with~\cite{symplectic} yield a complete list of symplectic criteria that are able to decide the symplectic type of $(E,E',p)$ whenever this is possible using only local information at a single prime~$\ell \neq p$.
\begin{theorem}\label{thm:summary}
Let $\ell$ and $p \geq 3$ be different primes. Let $(E,E',p)$ be as above.

Let $\cL_{(E,E',p)}$ be the set of primes such that a symplectic criterion exists for $E/\Q_\ell$ and~$E'/\Q_\ell$.

Then~$\ell \in \cL_{(E,E',p)}$ if and only if, after replacing $E$ and $E'$ by twists according
to Table~\ref{Table:TwistingLemmas} if necessary, exactly
one of the rows of Table~\ref{Table:CriteriaList} is satisfied. Moreover, the symplectic type of $(E,E',p)$ is determined by the
theorems listed under {\sc Criteria} in that row.
\end{theorem}

As an application, we determine the symplectic type of all triples $(E,E',p)$ to which our Theorems~\ref{thm:multiplicative} and~\ref{thm:mixedReduciton} apply, where $p \geq 5$ and $E$ and $E'$ are elliptic curves in LMFDB~\cite{lmfdb} with conductor $\leq 500 000$;
as expected, for $p \geq 7$, we obtained results consistent with~\cite{CremonaFreitas}. Somewhat surprisingly, all examples obtained with this computation involve irreducible mod~$p$ representations. For completeness, we describe the family of all elliptic curves over~$\Q$ having mod~$5$ representation of the form~$\left(\begin{smallmatrix}
  \chi_5 & * \\ 0 & 1
\end{smallmatrix} \right)$,
and use it to generate examples of application of our criteria in the reducible case.

\subsection{Acknowledgements}
We thank John Cremona and Diana Mocanu for various  discussions regarding the algorithm in Section~\ref{s:goodred} and its implementation, and for other comments and corrections.
This work was completed while Freitas and S\'anchez-Rodr\'iguez
were at the Max Planck Institute for Mathematics in Bonn. These two authors are grateful to the Max Planck Institute for Mathematics in Bonn for its hospitality and financial support.

\section{The existence of symplectic criteria}

%For $K$ a field of characteristic zero or a finite field of characteristic~$\neq p$
%we fix, for all primes~$p$, a primitive $p$-th root of unity $\zeta_p \in \overline{K}$.
%For an elliptic curve $E/K$,
%we donete by~$e_{E,p}$ the Weil pairing on~$E[p]$.
%We write $G_K = \Gal(\overline{K}/K)$.

\subsection{Auxiliary results} We recall \cite[Theorem~15]{symplectic} which we will need.

\begin{theorem} \label{thm:thm15}
Let $\ell$ and $p$ be different primes with $p \geq 3$.
Let $F \subset \Qbar_\ell$ be a field. Let
$E/F$ and $E'/F$ be elliptic curves with isomorphic $p$-torsion modules.
Then a symplectic criterion exists if and only if one of the following conditions holds:
\begin{itemize}
 \item[(A)] $\rhobar_{E,p}(G_F)$ is non-abelian;
 \item[(B)] $\rhobar_{E,p}(G_F)$ is generated, up to conjugation,
by $\left(\begin{smallmatrix} a & 1 \\ 0 & a \end{smallmatrix}\right)$ where $a \in \F_p^*$.
\end{itemize}
\end{theorem}

We recall also well known facts
about the Tate curve.

\begin{proposition} \label{P:theoryTate}
Let $\ell$ and $p \geq 3$ be different primes.
Let $E/\Q_\ell$ be an elliptic curve with potentially multiplicative reduction
and minimal discriminant~$\Delta_m(E)$.

(A) There is a symplectic basis of $E[p]$ in which the $p$-torsion representation is of the form
\begin{equation}
 \label{E:tate0}
 \rhobar_{E,p} \simeq \chi \otimes
 \begin{pmatrix}\chi_p & h \\ 0 & 1 \end{pmatrix},
\end{equation}
where $\chi_p$ is the mod~$p$ cyclotomic
character and $\chi$ is
character of order at most 2. Moreover, $E$ has additive reduction if and only if~$\chi$ is ramified and $E$ has split multiplicative reduction if and only if~$\chi = 1$.

(B) Assume $E$ has split multiplicative reduction. Let $q$ be the Tate parameter of~$E$. Then,
\begin{enumerate}
 \item Via Tate's uniformization, we can assume that~\eqref{E:tate0} is given in the basis is $\{\zeta_p, q^{1/p}\}/q^{\Z}$.
 \item We have $h|_{I_\ell} = 0 \iff \#\rhobar_{E,p}(I_\ell) = 1  \iff p \mid v(\Delta_m(E)) = -v(j_E)$.
 \item We have $\Q_\ell(E[p]) = \Q_\ell(\zeta_p,q^{1/p})$.
 \item For $\sigma \in G_{\Q_\ell(\zeta_p)}$, we have $\sigma(j_E^{1/p}) = \zeta_p^{-h(\sigma)} \cdot j_E^{1/p}$ for any choice of~$j_E^{1/p}$ a $p$-th root of $j_E$.
\end{enumerate}
\end{proposition}
\begin{proof}
Everything in this proposition follows from the well known theory of the Tate curve,
treated in detail
in~\cite[Chap. V]{SilvermanII};
in particular, see Proposition~6.1 and Exercise~5.13 in {\it loc. cit.}.
For the claim that~\eqref{E:tate0} holds in a symplectic basis see~\cite[\S2.4 and Remark~2.10]{FukudaWeilPairing}.

We will prove only (B) part (4).

Fix $\zeta_p \in \Qbar^*_\ell$ a $p$-th root of unity
and $q^{1/p} \in \Qbar^*_\ell$ a $p$-th root of~$q$. Consider the basis in~(1), so that $\rhobar_{E,p}$ is given by~\eqref{E:tate0} with $\chi = 1$.
For~$\sigma \in G_{\Q_\ell}$, the value $h(\sigma)$ is defined by $\sigma(q^{1/p}) = \zeta_p^{h(\sigma)}q^{1/p}$.

We have $q \cdot j_E = s^p$ in~$\Q_\ell$ by \cite[Lemma 10.3]{FKdegree}.
Thus $j_E^{1/p} = q^{-1/p} \cdot s$ is a $p$-th root of $j_E^{1/p}$ satisfying
$$\sigma(j_E^{1/p}) = \sigma(q^{-1/p} \cdot s) = \zeta_p^{-h(\sigma)} q^{-1/p} \cdot s = \zeta_p^{-h(\sigma)}j_E^{1/p}.$$
Finally, we show the value $-h(\sigma)$ is independent of the choice of~$j_E^{1/p}$ when $\sigma \in G_{\Q_\ell(\zeta_p)}$. Any other $p$-th root of $j_E$ is of the form
$r = \zeta_p^k \cdot j_E^{1/p}$.
Thus, for $\sigma \in G_{\Q_\ell(\zeta_p)}$, we have
\[ \sigma(r) = \sigma (\zeta_p^k \cdot j_E^{1/p}) = \zeta_p^k \cdot  \sigma(j_E^{1/p}) =  \zeta_p^k \cdot \zeta_p^{-h(\sigma)} \cdot j_E^{1/p}
= \zeta_p^{-h(\sigma)} \cdot r,
\]
as desired.
\end{proof}
\subsection{The case of potentially multiplicative reduction}
For an elliptic curve $E/\Q_\ell$ with potentially multiplicative reduction and $j$-invariant $j_E$ we define $\tilde{j}_E:=j_E\cdot \ell^{-v(j_E)}$.

We prove the existence theorem.

\begin{theorem}\label{thm:existenceTate}
Let $\ell$ and $p \geq 3$ be different primes. Let $E$ and $E'$ be elliptic curves
over~$\Q_\ell$ with isomorphic $p$-torsion modules.
Suppose that both curves have potentially multiplicative reduction.
Then a symplectic criterion exists if and only if either $p \nmid v(j_E)$ or
\[ p \mid v(j_E), \quad \ell \equiv 1 \pmod{p}  \quad \text{ and }
 \quad {\tilde{j}_E}^{\frac{\ell-1}{p}} \not\equiv 1 \pmod{\ell}.
\]
\end{theorem}
\begin{proof}
Let $\Q_\ell(\sqrt{d})$ be the field fixed the character~$\chi$ in~\eqref{E:tate0} when applying Proposition~\ref{P:theoryTate} to~$E$ or~$E'$ (we obtain the same $\chi$ as $E[p] \simeq E'[p]$). The existence of a symplectic criterion is invariant by taking simultaneous quadratic twists (see~\cite[Corollary~6]{symplectic}).
Then, after replacing $E$ by $dE$ and $E'$ by $dE'$, both curves have split multiplicative reduction and
\[
\rhobar_{E',p} \simeq \rhobar_{E,p} \simeq
 \begin{pmatrix} \chi_p & h \\ 0 & 1 \end{pmatrix} \quad \text{where} \quad
 \#\rhobar_{E,p}(I_\ell) = 1
 \iff
p \mid v(j_E) \iff h|_{I_\ell} = 0.
\]
Suppose a symplectic criterion exists and $p \mid v(j_E)$. We have $h|_{I_\ell} = 0$ and $\rhobar_{E,p}(G_{\Q_\ell})$ satisfies~(A) or (B) in Theorem~\ref{thm:thm15}. Thus $h(\Frob_\ell) \neq 0$ and $p \mid \#\rhobar_{E,p}(G_{\Q_\ell})$ in both cases. The simultaneous conditions
$p \mid \#\rhobar_{E,p}(G_{\Q_\ell})$ and $p \mid v(j_E)$ are compatible only with the second part of case 1.2) of~\cite[Theorem~3.1]{FKdegree} (note that $E$ has split multiplicative reduction so $-c_6(E)$ is a square in~$\Q_\ell$ so we are in case 1) of {\it loc. cit.}). Thus $\ell \equiv 1 \pmod{p}$ and ${\tilde{j}_E}^{\frac{\ell-1}{p}} \not\equiv 1 \pmod{\ell}$.

For the converse, note that if $p \nmid v(j_E)$ then the original criterion of Kraus--Oesterl\'e~\cite[\S4.3]{symplectic} exists an can be applied. If instead $p \mid v(j_E)$,
$\ell \equiv 1 \pmod{p}$ and ${\tilde{j}_E}^{\frac{\ell-1}{p}} \not\equiv 1 \pmod{\ell}$
then, again by~\cite[Theorem 3.1]{FKdegree},
we have $\rhobar_{E',p} \simeq \rhobar_{E,p} \simeq
 \left(\begin{smallmatrix} 1 & h \\ 0 & 1 \end{smallmatrix}\right)$ with $h(\Frob_\ell)\neq 0$. Thus case (B) of Theorem~\ref{thm:thm15} holds and a symplectic criterion exists.
\end{proof}

\subsection{The mixed reduction case}
Recall that for an elliptic curve $E/\Q_\ell$ with potentially good reduction,
the semistability defect~$e(E)$ of $E$ is the degree of the
minimal extension of~$\Q_\ell^{\text{un}}$ where
$E$ obtains good reduction. It is well known that $e(E) \in \{1, 2, 3, 4, 6, 8, 12, 24\}$.

\begin{proposition} \label{P:multTwist}
Let $\ell$ and~$p \geq 3$ be different primes. Let $E$ and $E'$ be elliptic
curves over~$\Q_\ell$ such that
$E[p]$ and $E'[p]$ are isomorphic $G_{\Q_{\ell}}$-modules.

Suppose that $E$ has potentially multiplicative reduction and $E'$ has potentially good
reduction.

Then,
after taking a simultaneous
quadratic twist of $E$ and $E'$ by some $d \in \Q_\ell$ if necessary,
we are in one of the following cases:
\begin{itemize}
 \item[(i)] $E'$ has good reduction, $E$ has split multiplicative reduction and
 $p \mid v(\Delta_m(E))$,
 where $\Delta_m(E)$ is the discriminant of a minimal model for~$E$.
 \item[(ii)] $E'$ has potentially good reduction with semistability defect $e(E')=3$, $p=3$ and $E$ has split multiplicative reduction.
 \end{itemize}
\end{proposition}
\begin{proof}
Let $\Q_\ell(\sqrt{d})$ be the field fixed by~$\chi$ in~\eqref{E:tate0} when applying Proposition~\ref{P:theoryTate} to~$E$. Then $W=dE$ has split multiplicative reduction and satisfies
\[
\rhobar_{W,p} \simeq
 \begin{pmatrix} \chi_p & h \\ 0 & 1 \end{pmatrix} \quad \text{where} \quad
 \#\rhobar_{W,p}(I_\ell) = 1
 \iff
p \mid v(\Delta_m(W)).
\]
Since $\chi_p$ is unramified at~$\ell$, we have that $\rhobar_{W,p}(I_\ell)$ is of order 1 or~$p$.

Write $W' = dE'$. This curve
has potentially good reduction.
From \cite[Lemma~11]{symplectic}, we have $W[p] \simeq W'[p]$ as $G_{\Q_{\ell}}$-modules. In particular,
$\# \rhobar_{W,p}(I_\ell) = \# \rhobar_{W',p}(I_\ell)= e(W')$.

We split into cases:

(i) Suppose $\# \rhobar_{W,p}(I_\ell) = 1$.
Then $p \mid v(\Delta_m(W))$ and $e(W') = 1$ (i.e. $W'$ has good reduction)

(ii) Suppose $\# \rhobar_{W,p}(I_\ell) = p$.
Since $e(W') \in \{1, 2, 3, 4, 6, 8, 12, 24\}$ we conclude $p= e(W')=~3$.
\end{proof}

We now prove the existence of a symplectic criterion.

\begin{theorem}\label{thm:existence}
Let $\ell$ and $p \geq 3$ be different primes. Let $E$ and $E'$ be elliptic curves
over~$\Q_\ell$ with isomorphic $p$-torsion modules.
Suppose that $E$ has potentially multiplicative reduction and $E'$ has potentially good
reduction.
A symplectic criterion exists if and only if either
\begin{itemize}
 \item[(a)] $p=3$ and $E'$ has semistability defect $e(E') \in \{3,6 \}$; or
 \item[(b)] $\ell \equiv 1 \pmod{p}$, $v(j_E) \equiv 0 \pmod{p}$ and
 ${\tilde{j}_E}^{\frac{\ell-1}{p}} \not\equiv 1 \pmod{\ell}$ where ${\tilde{j}_E} = j_E \cdot \ell^{-v(j_E)}$.
\end{itemize}
\end{theorem}

\begin{proof}
If (a) holds then a symplectic criterion is described in \cite[\S 4.4]{symplectic}.
Suppose (b) holds.  Then $\chi_p = 1$ and
$\rhobar_{E,p} \simeq
 \left(\begin{smallmatrix} \chi & h \\ 0 & \chi \end{smallmatrix}\right)$ with $h \neq 0$ by Proposition~\ref{P:theoryTate} and \cite[Theorem~3.1]{FKdegree}. Hence case~(B) of \cite[Theorem 15]{symplectic} is satisfied thus a symplectic criterion exists.

The existence of a symplectic criterion is invariant by taking simultaneous quadratic twists (see~\cite[Corollary~6]{symplectic}).
Therefore, from Proposition~\ref{P:multTwist},
after twisting both curves by some~$d$ if needed,
we can assume that either (a) holds with $e(E')=3$ or
(b') $E'$ has good reduction, $E$ has split multiplicative reduction and
$v(\Delta_m(E)) = - v(j_E) \equiv 0 \pmod{p}$.

Suppose now that a symplectic criterion exists. We can assume (b') holds, otherwise (a) holds and we are done. From $E'[p] \simeq E[p]$ and Proposition~\ref{P:theoryTate} applied to~$E$ we obtain
\begin{equation} \label{E:tate1}
\rhobar_{E',p} \simeq \rhobar_{E,p} \simeq
 \begin{pmatrix} \chi_p & h \\ 0 & 1 \end{pmatrix} \quad \text{ with } \quad h|_{I_\ell} = 0.
\end{equation}
Since $E'$ has good reduction, the image of $\rhobar_{E',p}$ is cyclic, so
$\#\rhobar_{E',p}(G_{\Q_\ell}) = \#\rhobar_{E,p}(G_{\Q_\ell})$ divides~$p(p-1)$.
As a symplectic criterion exists, we are in case (B) of \cite[Theorem~15]{symplectic}, therefore $\#\rhobar_{E,p}(G_{\Q_\ell}) = p$.
Thus $\chi_p = 1$ and $h \neq 0$. In particular, $\ell \equiv 1 \pmod{p}$.
Finally, since $j_{E}$ is invariant by quadratic twists, we obtain ${\tilde{j}_E}^{\frac{\ell-1}{p}} \not\equiv 1 \pmod{\ell}$ from~\cite[Theorem~3.1]{FKdegree}, proving~(b).
\end{proof}

\section{Proof of Theorem~\ref{thm:multiplicative}}

%Since $p \mid \ell-1$, we can choose a primitive $p$-th root of unity $\zeta_p \in \F_\ell$. We denote also by~$\zeta_p$ its lift in $\Z_\ell^\times$.

Fix $\tilde{j}_i^{1/p} \in \Qbar_\ell$ a $p$-th root of $\tilde{j}_i \in \Z_\ell^\times$ and set
$j_{E_i}^{1/p} := \tilde{j}_i^{1/p} \ell^{\frac{v(j_E)}{p}}$.

Let $K / \Q_\ell$ be the unique unramified extension of degree~$p$.

From Proposition~\ref{P:theoryTate} and~\cite[Theorem~3.1]{FKdegree} it follows that, for $i=1,2$, the $p$-torsion field cut out by $\rhobar_{E_i,p}$ is unramified and of degree~$p$.
Then $\ker \rhobar_{E_1,p} = \ker \rhobar_{E_2,p}$
fixes~$K$.

Let $q_i$ be the Tate parameter of~$E_i$. Since $j_{E_i} \cdot q_i = s^p$ in $\Q_\ell$ by~\cite[Lemma 10.3]{FKdegree},
we have $K = \Q_\ell(q_i^{1/p}) = \Q_\ell(j_{E_i}^{1/p}) = \Q_\ell(\tilde{j}_i^{1/p})$ by Proposition~\ref{P:theoryTate}.
Let $\sigma := \Frob_\ell \in \Gal(K/\Q_\ell)$.

From Proposition~\ref{P:theoryTate} we have
$\rhobar_{E_i,p}(\sigma) = \left(\begin{smallmatrix}
  1 & -h_i \\
  0 & 1
 \end{smallmatrix}\right)$
in the symplectic basis $\{\zeta_p, q_i^{1/p}\} / q_i^\Z$,
where~$h_i$ is defined by $\sigma(j_{E_i}^{1/p}) = \zeta_p^{h_i} \cdot j_{E_i}^{1/p}$.
Since $\rhobar_{E_1,p}(\sigma) = g \cdot \rhobar_{E_2,p}(\sigma) \cdot g^{-1}$
for
$g = \left(\begin{smallmatrix}
  1 & 0 \\
  0 & h_2/h_1
\end{smallmatrix}\right)$, we conclude $E[p] \simeq E'[p]$ as $G_{\Q_\ell}$-modules; the second part of the first statement now follows from Theorem~\ref{thm:existenceTate}.
%Note that the centralizer of $\rhobar_{E_i,p}(G_{\Q_\ell})$ in $\GL_2(\F_p)$ is the group of matrices of the form~$\left(\begin{smallmatrix}
%  \lambda & b \\
%  0 & \lambda
%\end{smallmatrix}\right)$ with $\lambda \ne 0$.
Moreover, from \cite[Lemma 6]{symplectic} we conclude that $E_1[p]$ and~$E_2[p]$ are symplectically isomorphic if and only if $h_1/h_2$ is a square mod~$p$.

Define $\tilde{h}_i$ by the relation
$\sigma(\tilde{j}_i^{1/p}) = \zeta_p^{\tilde{h}_i}\tilde{j}_i^{1/p}$.
We have
\[
\zeta_p^{h_i} \cdot j_{E_i}^{1/p} = \sigma(j_{E_i}^{1/p}) = \sigma(\tilde{j}_i^{1/p})\cdot\ell^\frac{v(j_E)}{p} = \zeta_p^{\tilde{h}_i} \cdot \tilde{j}_i^{1/p} \ell^\frac{v(j_E)}{p} =  \zeta_p^{\tilde{h}_i} \cdot j_{E_i}^{1/p}
\]
thus $h_i = \tilde{h}_i$ in $\F_p$. Since $K/\Q_\ell$ is unramified, $\ell \in \calO_K$ is a uniformizer. We have $\tilde{j}_i^{1/p} \in \calO_K^\times$ and,
by definition of $\sigma = \Frob_\ell$, we obtain
\[
 \sigma(\tilde{j}_i^{1/p}) \equiv \tilde{j}_i^{\ell/p} \pmod{\ell \calO_K}
 \implies
 \tilde{j}_i^{1/p} \zeta_p^{h_i} \equiv \tilde{j}_i^{ \ell/p} \pmod{\ell \calO_K}
\implies \zeta_p^{h_i} \equiv \tilde{j}_i^{ \frac{\ell-1}{p}} \pmod{\ell}
 \]
therefore $h_i = \text{Log}_{\zeta_p}(\tilde{j}_i^{\frac{\ell-1}{p}} \pmod{\ell})$ in $\F_\ell$, proving the second statement.

%Finally, since $E_1[p] \simeq E_2[p]$ we are under the assumptions of Theorem~\ref{thm:existenceTate}, hence a symplectic criterion exists, proving the last statement.

\section{Proof of Theorem~\ref{thm:mixedReduciton}}

We have $\chi_p = 1$. Therefore, from assumption (1), Proposition~\ref{P:theoryTate}
and \cite[Theorem 3.1]{FKdegree}, we can choose a  symplectic basis for $E[p]$ such that
\[
\rhobar_{E,p} (\Frob_\ell) =
 \begin{pmatrix} 1 & -h \\ 0 & 1 \end{pmatrix} \quad \text{ with } \quad h|_{I_\ell} = 0, \quad h = \text{Log}_{\zeta_p}(\tilde{j}_E^{\frac{\ell-1}{p}} \pmod{\ell}) \neq 0,
\]
where the formula for~$h$ follows as in the proof of Theorem~\ref{thm:multiplicative}.

Assumption (2) and \cite[Lemma 8]{symplectic} give
that $\rhobar_{E,p}(\Frob_\ell)= \left(\begin{smallmatrix}
  1 & 1 \\
  0 & 1
\end{smallmatrix}\right)$ in some basis $\{P_1,P_2\}$. If $e_E(P_1,P_2) = \zeta_p^{h'}$ and
$ h'  \cdot  k \equiv 1 \pmod{p}$ then $\{ k \cdot P_1, P_2\}$ is a symplectic basis
in which
\[
\rhobar_{E',p} (\Frob_\ell) =
 \begin{pmatrix} 1 & h' \\ 0 & 1 \end{pmatrix},
\]
proving~(i).
We have $\rhobar_{E,p}(\Frob_\ell) = g \cdot \rhobar_{E',p}(\Frob_\ell) \cdot g^{-1}$
for
$g = \left(\begin{smallmatrix}
  1 & 0 \\
  0 & -h'/h
\end{smallmatrix}\right)$.
Since $E'$ has good reduction, $\rhobar_{E',p}$ is also unramified. Therefore,
both representation fix the unique unramified extension of~$K / \Q_\ell$ of degree~$p$,
therefore $\ker \rhobar_{E_1,p} = \ker \rhobar_{E_2,p}$
hence $E[p]$ and $E'[p]$ are isomorphic $G_{\Q_\ell}$-modules; the second part of the first statement now follows from Theorem~\ref{thm:existence}.
Finally, conclusion (ii) follows from \cite[Lemma 6]{symplectic} since $(\det(g) / p) =1$ if and only if $-h/h'$ is a square~mod~$p$.
%Finally, since we have proved $E[p] \simeq E'[p]$ we are under the assumptions of Theorem~\ref{thm:existence}; moreover, hypothesis (1) together with $\ell \equiv 1 \pmod{p}$ is precisely (b) of Theorem~\ref{thm:existence}, hence a symplectic criterion exists, proving the third statement.

\section{The case of good reduction}
 \label{s:goodred}

The available symplectic criteria for the case where both curves have good reduction are Theorem~12~and Theorem~16 in \cite{symplectic}, but while the former is very easy to apply in practice, the latter is impractical. This is because we don't know how to determine the value of~$s$ in its statement, more precisely,
we do not know the symplectic type of the basis giving the matrix description~(7.2) in {\it loc. cit}. To circumvent this issue, Theorem~\ref{thm:goodred} is an alternative formulation of the criterion, and we provide a {\tt Magma} implementation of it as well. Compared to~\cite[Theorem~16]{symplectic} the disadvantage is that we cannot use~\cite[Theorem~2]{Centeleghe} to quickly compute the action of Frobenius, but the upside is that we know the symplectic type of all basis.

\begin{proof}[Proof of Theorem~\ref{thm:goodred}] From the assumption $p \mid \# \rhobar_{E,p}(\Frob_\ell)$, Corollary 4 and Proposition~5 in \cite{symplectic}, we know that a symplectic criterion exists; thus the last statement follows.
Since $E$ has good reduction, $\rhobar_{E,p}(G_{\Q_\ell}) \subset \GL_2(\F_p)$ is cyclic and generated by $\rhobar_{E,p}(\Frob_\ell)$, hence we are in case~(B) of Theorem~\ref{thm:thm15}. It follows that there are basis of $E[p]$ and $E'[p]$ where the action of $\Frob_\ell$
is given by~$\left(\begin{smallmatrix}
  a_\ell/2 & 1 \\
  0 & a_\ell/2
\end{smallmatrix}\right)$.
Replacing the first element of each basis by an appropriate multiple
(depending on the curve), we obtain symplectic basis and Frobenius matrices in the desired form, proving the first statement (see the proof of Theorem~\ref{thm:mixedReduciton} for details where this argument was applied to~$E'$). Finally, the matrix $g = \left(\begin{smallmatrix}
  1 &  0\\
  0 & h'/h
\end{smallmatrix}\right)$ has determinant $h'/h$ and
conjugates $\rhobar_{E',p}$ into $\rhobar_{E,p}$, so \cite[Lemma~6]{symplectic}
completes the proof.
\end{proof}

For large~$p$, applying Theorem~\ref{thm:goodred} in practice can become computationally expensive, because it involves finding the splitting field of the $p$-division polynomial which has degree $O(p^2)$. To make this more accessible, we have implemented an algorithm in {\tt Magma} that makes use of all the theorem assumptions to reduce the required calculations, which we now summarize.

First observe that to determine if $h'/h$ is a square modulo~$p$ we only need to know the Legendre symbols $(h/p)$ and~$(h'/p)$. In view of this, we will describe an algorithm that given an elliptic curve $E$ as in the statement of Theorem~\ref{thm:goodred} determines~$(h/p)$.
This method can then be applied to $E$ and $E'$ to determine if $h'/h$ is a square. Further, it can also be applied to find~$(h'/p)$,
where~$h'$ is as in Theorem~\ref{thm:mixedReduciton}, as we will do in the examples of Section~\ref{s:examples}.

Fix $\zeta_p \in \overline{\F}_\ell$ a $p$-th root of unity and denote also by $\zeta_p$ its lift in $\Qbar_\ell$.

Since $E/\Q_\ell$ has good reduction, the reduction morphism $E[p] \to \oE[p]$ is a Galois equivariant isomorphism that preserves the Weil pairing, where $\Ebar / \F_\ell$ is the residual curve of $E/\Q_\ell$. Therefore, given a symplectic basis of $\Ebar[p]$ we can lift it to a symplectic basis of $E[p]$ and, in those basis, we have the equality $\rhobar_{E,p}(\Frob_\ell) = \rhobar_{\oE,p}(\Frob_\ell)$ in $\GL_2(\F_p)$, where we are using $\Frob_\ell$ to denote both the Frobenius element in $G_{\F_\ell}$ and a lift in $G_{\Q_\ell}$.
Hence we can work with $\oE / \F_\ell$ and all the calculations required to find $(h/p)$ are over finite fields.

We see from the theorem that $\rhobar_{\Ebar,p}(G_{\F_\ell})$ can be conjugated into the Borel and the curve $\Ebar$ admits exactly one $p$-isogeny $\varphi$ over $\F_\ell$.

We determine the kernel of $\varphi$ by finding the unique degree $(p-1)/2$ factor of the $p$-division polynomial of $\oE$ and
construct the $p$-isogeny $\varphi:\oE \to W:=\oE/\ker(\varphi)$.
Take any non-trivial point $P_W \in \ker(\hat{\varphi})$, where $\hat{\varphi}$ is the dual isogeny of~$\varphi$, and let $Q$ be any point in the preimage of~$\varphi^{-1}(P_W)$ or minus any such point.
Let $P$ be any non-trivial point in $\ker \varphi$. A short argument, shows that $Q$ has order $p$ and it is linearly independent to $P$.
Thus $\{P,Q\}$ form a basis for $\oE[p]$. This approach involves finding roots of polynomials of degree $O(p)$ over the field of definition of~$P$ to find the $x$-coordinate of~$Q$.

Using the Weil pairing, we check whether $\{P, Q\}$ is a symplectic basis (up to multiplying both $P,Q$ by a scalar); if it is not, then replace $Q$ by $\epsilon Q$ where $\epsilon$ is a non-square modulo $p$, so that the resulting basis $B_{\oE} = \{P, Q\}$ satisfies that is $e_{\oE}(P,Q) = \zeta_p^{k^2}$.

By our assumptions $P^{\Frob_\ell} = \frac{a_\ell}{2} \cdot P$ and $\Tr(\Frob_\ell) = a_\ell$.
Thus, in the basis $B_{\oE}$, we have
\[
\rhobar_{\oE,p}(\Frob_\ell)=\left(\begin{smallmatrix}
				 a_\ell/2 & h \\
				0 & a_\ell/2\end{smallmatrix}\right)
\]
for some $h \ne 0$ as predicted by Theorem~\ref{thm:goodred}. Next, we have
\begin{equation}
\label{eq:WP}
\zeta_p^\alpha = e_{\oE}( Q^{\Frob_\ell },Q)= e_{\oE}(h P + (a_\ell/2)Q, Q)= e_{\oE}(P,Q)^h \cdot  e_{\oE}(Q,Q)^{a_\ell/2}  = \zeta_p^{k^2h},
\end{equation}

where we used bilinearity and non-degeneracy of the Weil pairing,
therefore
\begin{equation}
 \label{eq:WP2}
(h / p) = (k^2h/p) = (\alpha / p).
\end{equation}

\begin{example}\label{ex:211}
Let $\ell=73$ and $p=211$. Consider the elliptic curves over $\F_\ell$
\[E_1 : y^2=x^3 - 3x + 6\quad \text{ and } \quad E_2 : y^2=x^3+55x+5.\]
Using {\tt Magma} implementation of \cite[Theorem 2]{Centeleghe} we quickly check that the action of Frobenius on $E_i[p]$ is given by
$\left(\begin{smallmatrix}
				 110 & 0 \\
				1 & 110\end{smallmatrix}\right)$
hence it has the correct shape for Theorem~\ref{thm:goodred}. The issue is that we don't know the symplectic type of the basis giving this matrix representation.

The $p$-division polynomial of $E_1$ factors over $\F_\ell$ into the product of 5 irreducible polynomials of degree~$21$ and 5 of degree~$4431$.
The point~$P$ lies in $\F_{\ell^{42}}$ and $Q$ in $\F_{\ell^{8862}}$, the latter field being the $p$-torsion field of $E_1$ (and hence of~$E_2$).
Checking if $\{P,Q\}$ is a symplectic basis (up to a scalar) is quick and we replace $Q$ by $-Q$ if it is not, since $(-1/p)=-1$ as $p\equiv 3 \pmod{4}$.
Finally, we compute $(h_1/p) = 1$ following~\eqref{eq:WP} and~\eqref{eq:WP2}. Applying the same method to~$E_2$ yields $(h_2/p) = 1$, therefore $E_1[p]$ and~$E_2[p]$ are symplectically isomorphic by Theorem~\ref{thm:goodred}.

Note that the Weil pairing computations with $E_1$ and~$E_2$ need to be done with respect to the same fixed $\zeta_p \in \overline{\F}_\ell$. To achieve this we first construct the field $\F_\ell(\zeta_p)$ and construct all the fields in the computation above as extensions of it, so that the final Weil pairing computations automatically comes in terms of~$\zeta_p$.
The computing time for this example is approximately 2.5 hours; we remark that naively searching for roots of the $p$-division polynomial of~$E_1$ had run for two weeks without progress.
\end{example}

\begin{remark}\label{rm:Magma}
The main difficulty to apply Theorem~\ref{thm:goodred} in practice is defining the $p$-torsion points when~$p$ is large, since they lie in an extension of
$\F_\ell$ of degree up to~$p^2-p$. The \texttt{Magma} software usually constructs finite fields via Conway polynomials, which is slow for large finite fields, for example, the field $\F_{73^{8862}}$ involved in Example~\ref{ex:211} was taking about~8h to initiate; moreover, finding roots of the $p$-division polynomial in such a large field was impractical. Instead, we can define the relevant finite field extensions as relative extensions using irreducible factors of the $p$-division polynomial; this has two significant advantages: it initiates the field faster and automatically obtains the $x$-coordinate of~$Q$ as its primitive element; see~\cite{git} for full implementation details.
\end{remark}

\section{Proof of Theorem~\ref{thm:summary}}

Let $(E,E',p)$ be as in the statement.

We have $(E/\Q_\ell)[p] \simeq (E'/\Q_\ell)[p]$ as $G_{\Q_\ell}$-modules for all primes~$\ell$.
Moreover, if any of the theorems under {\sc Criteria} in Table~\ref{Table:CriteriaList} applies then its output coincides with the symplectic type of~$(E,E',p)$ by \cite[Corollary 2]{symplectic}.

Suppose $E/\Q_\ell$ and $E'/\Q_\ell$ satisfy a row of Table~\ref{Table:CriteriaList}. Then the theorem under {\sc Criteria} in that row applies; in particular, a symplectic criterion exists, so $\ell \in \cL_{(E,E',p)}$. Recall that the existence of a symplectic criterion is invariant
by simultaneous quadratic twist. Thus, if $E/\Q_\ell$ and~$E'/\Q_\ell$ satisfy a row in Table~\ref{Table:CriteriaList} after a twist prescribed by Table~\ref{Table:TwistingLemmas},
then $\ell \in \cL_{(E,E',p)}$.

Suppose now that $\ell \in \cL_{(E,E',p)}$, that is, a symplectic criterion exists at~$\ell$.

Note that, for each prime~$\ell$, the curves  $E/\Q_\ell$ and $E'/\Q_\ell$ satisfy one of the following :
\begin{itemize}
 \item[(i)] both have potentially multiplicative reduction; or
 \item[(ii)] $E/\Q_\ell$ has potentially multiplicative reduction and $E'/\Q_\ell$ has potentially good reduction; or
 \item[(iii)] both curves have potentially good reduction.
\end{itemize}

Observe that applying a simultaneous quadratic twist by some~$d \in \Q_\ell$ obtained by applying one of the twisting lemmas in Table~\ref{Table:TwistingLemmas} to (say)~$E$ changes the reduction type of $E'$ in the unique way that
will preserve the isomorphism $(E/\Q_\ell)[p] \simeq (E'/\Q_\ell)[p]$ after twist.

Assume (i). Applying a common quadratic twist prescribed in \cite[Lemma 5]{symplectic} we can assume both
$E/\Q_\ell$ and $E'/\Q_\ell$ have split multiplicative reduction. Now, from Theorem~\ref{thm:existenceTate} we see that either row 11 or 12 in Table~\ref{Table:CriteriaList} is satisfied.

Assume (ii). Then either (a) or (b) in Theorem~\ref{thm:existence} is satisfied. If (a) holds with $e'=6$,
by taking a quadratic twist prescribed by any of \cite[Lemmas~1, 2 or 5]{symplectic}, we can assume that
$E/\Q_\ell$ has multiplicative reduction and $E'/\Q_\ell$ has $e'=3$; if (a) holds with $e'=3$ then~$E/\Q_\ell$ already has multiplicative reduction. In both cases row 13 of Table~\ref{Table:CriteriaList} applies. If (b) holds, after the twist given by \cite[Lemmas~5]{symplectic}, we can assume that $E/\Q_\ell$ has split multiplicative reduction and $E'/\Q_\ell$ good reduction, hence row 14 applies.

Assume (iii). This follows from case 3 in the proof of \cite[Theorem~19]{symplectic}.\qed

\begin{small}
 \begin{table}[htb]
\renewcommand{\arraystretch}{1.25}
\begin{tabular}{|c|c|c|c|} \hline
{\sc reduction type of} $E$, $E'$ & {\sc prime} $\ell \neq p \geq 3$& {\sc Extra conditions} & {\sc criteria}  \\\hline
good ($e=1$)  & $(\ell/p) = 1$       &  $p \mid \Delta_\ell, \; p \nmid \beta_\ell, \; \Ebar = \Ebar'$ & Theorem~12   \\
good ($e=1$)  & $(\ell/p) = 1$       &  $p \mid \Delta_\ell, \; p \nmid \beta_\ell$ & Theorems~\ref{thm:goodred}, 16   \\
pot. good $e = 3$ &$\ell \equiv 2 \pmod{3}$ &      & Theorems~1,2   \\
pot. good $e = 3$ &$\ell \equiv 1 \pmod{3}$ &   $p=3$   & Theorem~3  \\
pot. good $e = 3$& $\ell = 3$               & $\tilde{\Delta} \equiv 2 \pmod{3}$ & Theorem~4   \\
pot. good $e=4$         &$\ell \equiv 3 \pmod{4}$ &                                     & Theorem~5   \\
pot. good $e=4$         &$\ell = 2$               &  $\tilde{c}_4 \equiv 5\tilde{\Delta} \pmod{8}$& Theorem~6  \\
pot. good $e=8$         &$\ell = 2$               &    & Theorems~8,9   \\
pot. good $e=12$        &$\ell = 3$               &     & Theorems~10,~11   \\
pot. good $e=24$        &$\ell = 2$               &    & Theorem~7   \\ \hline
multiplicative &   any             &  $p \nmid v(j_E)$    & Theorem~13   \\
split multiplicative &   $\ell \equiv 1 \pmod{p}$             &  $p \mid v(j_E), \; \tilde{j}^{\frac{\ell-1}{p}} \not\equiv 1 \pmod{\ell}$    & Theorem~\ref{thm:multiplicative}   \\ \hline
$E$ multiplicative, $e'=3$ & $\ell \neq 3$               &  $p=3$ & Theorem~14  \\
\multirow{2}{*}{$E$ split multiplicative, $E'$ good} & \multirow{2}{*}{any} & $p \mid v(j_E)$, \; $\tilde{j}^{\frac{\ell-1}{p}} \not\equiv 1 \pmod{\ell}$ & \multirow{2}{*}{Theorem~\ref{thm:mixedReduciton}} \\
                                                     &                      & $p\mid \#\rhobar_{E',p}(\Frob_\ell)$                                        & \\
\hline
\end{tabular}
\caption{The theorems under {\sc criteria} are all in~\cite{symplectic} except for Theorem~\ref{thm:multiplicative},~\ref{thm:mixedReduciton} and~\ref{thm:goodred}.
We assume $(E/\Q_\ell)[p] \simeq (E'/\Q_\ell)[p]$ as $\Gal(\Qbar_\ell/ \Q_\ell)$-modules;
$e$ and $e'$ are the semistability defects of $E$ and $E'$; the quantities
$\Delta_\ell$, $\beta_\ell$, $\tilde{c}_4$, $\tilde{\Delta}$ and $\tilde{j}$
are attached to~$E$; see \cite[\S4]{symplectic} for all definitions.}
\label{Table:CriteriaList}
\end{table}
\end{small}

\begin{small}
 \begin{table}[htb]
 \renewcommand{\arraystretch}{1.25}
\begin{tabular}{|c|c|c|c|} \hline
\multicolumn{3}{|c|}{{\sc Twisted reduction types for $E$, $E'$}} & {\sc Twisting Lemmas in \cite{symplectic}} \\ \hline
\multicolumn{3}{|c|}{$e=e'=2$ (twist of good)}  & Lemmas~3,~4 \\
\multicolumn{3}{|c|}{$e=e'=6$ (twist of $e=e'=3$)} & Lemmas~1,~2  \\
\multicolumn{3}{|c|}{$E$, $E'$ pot. multiplicative} & Lemma~5  \\
\multicolumn{3}{|c|}{$E$ pot. multiplicative, $e'=6$}  & Lemmas~1,~2,~5  \\
\multicolumn{3}{|c|}{$E$ pot. multiplicative, $e'=2$}  & Lemma~5  \\
\hline
\end{tabular}
\caption{Description of the twists required to reduce every possible combination of reduction types to one of the cases in Table~\ref{Table:CriteriaList}.}
\label{Table:TwistingLemmas}
\end{table}
\end{small}

\section{Applications}
\label{s:examples}

In~\cite{CremonaFreitas}, the first author and John Cremona, determined the symplectic type of all triples $(E,E',p)$ where $p \in \{7,11,13,17\}$ and $E, E'$ are elliptic curves over~$\Q$ of conductor $\leq 500000$.

In this section, we will first extend their search by finding all triples $(E,E',5)$ in the same conductor range. We note that, if $(E_1,E_2,5)$ is in our list, then its quadratic twist
$(dE_1,dE_2,5)$ is also in the list when both $dE_i$ have conductor $\leq 500000$; moreover, triples obtained from~ $(E_1,E_2,5)$ by
replacing $E_1$ or $E_2$ by an isogenous curve are also included;
only triples where $E_1$ and $E_2$ are isogenous were not included.
Since the lists from~{\it loc. cit.} contain only one representative among triples related by quadratic twist, to uniformize them with our list for $p=5$, we also detirmine the expanded lists for $p \geq 7$.
Secondly, we find the sublists of triples that satisfy Theorem~\ref{thm:multiplicative} or Theorem~\ref{thm:mixedReduciton}
at some prime~$\ell$ and determine their symplectic type using our criteria;
we did not search for when Theorem~\ref{thm:goodred} applies as there are infinitely many primes of good reduction.

All the lists together with code to apply the criteria can be found in~\cite{git};
we also provide the code to search for congruences in view of future expansions of the LMFDB database.

\subsection{Finding congruences mod~$p$}
To search for mod $p \geq 5$ congruences in the LMFDB database, we follow a variation of the strategy used in ~\cite[\S 3]{CremonaFreitas} for $p \geq 7$. In particular, we need to accommodate for the fact that an  elliptic curve over~$\Q$ can be 5-isogenous over~$\Q$ to two different curves, and we do not discard triples related by quadratic twist.

We summarize the main steps of our calculations:

1) Partition the set of isogeny classes of elliptic curves in the LMFDB into subsets $S$, such that whenever two curves have mod~$p$ representations with isomorphic semisimplifications, their isogeny classes belong to the same subset $S$, but not necessarily conversely. Indeed, we group the isogeny classes by their traces of Frobenius in the first 50 primes $> 500 000$ (this guarantees all curves have good reduction at these primes); it is conceivable that 2 isogeny classes could agree at these 50 primes and disagree in another prime, but it turns out this does not occur, otherwise it would be detected in a later step. We ignore all subsets of size~1.

2) Separate the subsets resulting from 1) into those corresponding to irreducible representations from corresponding to reducible representations; the latter occur only for $p=5,7$.

3) For each~$S$ corresponding to irreducible representations, and each isogeny class in $S$, which is represented by a single curve $E$ in~$S$, saturate the set $S$ by adding to~$S$ all the curves in the isogeny class of~$E$. For each~$S$, pick all the (non-ordered) pairs $(E,E')$ of non-isogenous curves in~$S$.
The next step will show this is precisely the set of all irreducible mod~$p$ congruences we are looking for.

4) For each pair $(E,E')$ computed in 3) we prove that $\rhobar_{E,p} \simeq \rhobar_{E',p}$ by comparing traces of Frobenius up to the `Sturm' bound given by~\cite[Proposition~4]{KO}.
First we minimize the set of pairs we are required to do this computation: (i) we only need to consider pairs up to quadratic twist, i.e., we pick one pair among all the pairs of the form $(dE,dE')$; and (ii) since we are comparing irreducible mod~$p$ representations, isogenous curves are related by isogenies of degree coprime to~$p$ so, among the pairs obtained in (i), we only need to consider one pair $(E,E')$ from all the pairs where $E$ and $E'$ belong to fixed isogeny classes.
If we find that $\rhobar_{E,p} \not\simeq \rhobar_{E',p}$ for some pair in~$S$, then we can further partition~$S$, and discard the resulting subsets of size~1; this did not occur, hence $\rhobar_{E,p} \simeq \rhobar_{E',p}$ for all pairs computed in the previous step.

\begin{remark}\label{rem:mod5}
The above computation took about 11 hours for all pairs with $p\geq 7$ and we estimate more than 2 months for~$p=5$ due to the 86530 pairs modulo~$5$ some of which giving rise to very large bounds; in fact, there are 30 pairs whose bound is around~$6 \cdot 10^9$ and there is no simple way in {\tt Magma} to generate the list of primes up to this number. Alternatively, we can use  the following much faster argument. Given an irreducible pair $(E,E')$ satisfying $(j(E),j(E')) \neq (0,0),(1728,1728)$ for which we want to prove $\rhobar_{E,5} \simeq \rhobar_{E',5}$, we consider the parametrization~$E_t^+$ of all elliptic curves with 5-torsion symplectically isomorphic to $E[5]$ given in~\cite{RS95} and the parametrization~$E_t^-$ of the elliptic curve with 5-torsion anti-symplectically isomorphic
to $E[5]$ from~\cite[Theorem 5.8]{fisherp5}. Let $j^\pm(t)$ be the $j$-invariant of $E_t^\pm$ and compute the rational roots of the numerators
of~$j^\pm(t)-j(E')$; if $E[5] \simeq E'[5]$ then one of the roots obtained must be the value $t'$ giving rise to $E'$ in one of the families $E_t^\pm$.
We checked that, for all mod 5 irreducible pairs computed in the previous step, we always find such a~$t'$ establishing simultaneously the isomorphism $E[5]\simeq E'[5]$ and its symplectic type; this takes about 6.5 hours. Finally, for each pair with $(j(E),j(E')) = (0,0)$ or~$(1728,1728)$ we conclude that $E[5]\simeq E'[5]$ by applying Theorem~1.2 and Corollary~2.5 in~\cite{CremonaFreitas}.
\end{remark}

By the end of 4), we determined the irreducible triples~$(E,E',p)$,  obtaining 396516 for~$p=5$, 39386 for~$p=7$, 848 for~$p=11$, 227 for~$p=13$ and 8 for $p=17$.

The reducible triples occur for $p=5,7$ and finding them requires additional steps.

5) For all $E$ in a reducible subset~$S$, we have
$\rhobar_{E,p} \simeq \left(\begin{smallmatrix}
  \chi & * \\
  0 & \chi'
\end{smallmatrix}\right)$,
hence $\rhobar_{E,p}(G_\Q)$ satisfies case (A) of Theorem~\ref{thm:thm15} if and only if $* \neq 0$ and $\chi \neq \chi'$ and case (B) if and only if $* \neq 0$ and $\chi = \chi' $; since $\det \; \rhobar_{E,p}$ is the mod~$p$ cyclotomic character which is not a square of a character of~$G_\Q$, the latter case cannot occur. Observe that for $p \geq 7$ we always have $* \neq 0$
by Corollary~3 and Proposition~2 in~\cite{symplectic} but for $p=5$ we may have $*=0$.
For each $E \in S$,  we include in $S$ all curves $E'$ in its isogeny class satisfying $\rhobar_{E',p} \simeq \left(\begin{smallmatrix}
  \chi & * \\
  0 & \chi'
\end{smallmatrix}\right)$
with $* \neq 0$; if the initial $E$ has $* = 0$ we remove it from the list.

6) The sets $S$ obtained in 5) satisfy that, for all $E_1, E_2 \in S$, assumption (ii) of \cite[Theorem~3.6]{CremonaFreitas}
is satisfied and assumption~(i) likely holds; for $p=7$
we prove that it actually holds applying \cite[Proposition~4]{KO} as in 4) but for $p=5$ this is very time consuming, as there are around 5500 pairs including variou with bound $>10^{10}$. The following argument replaces this computation. First, we pick
any curve $E \in S$ and compute the field~$F$ fixed by its isogeny character~$\chi$ and the field~$F'$ fixed by the isogeny character~$\chi'$ of the unique elliptic curve $5$-isogenous to~$E$ (which is in~$S$), and check that $F \neq F'$.
Second, for each curve in $S$ we test whether the field fixed by its isogeny character is isomorphic to $F$ or~$F'$ and group
them into sets $S(F)$ and $S(F')$, respectively; it is possible that some curves would yield a field not isomorphic to $F$ neither~$F'$ but this does not occur, hence we divided~$S$ into the union of the disjoint sets $S(F)$ and~$S(F')$.
Now the following lemma gives $S(F) = S(\chi)$ and $S(F') = S(\chi')$, where
$S(\chi)$ and $S(\chi')$ are the subsets of $S$ of curves having isogeny character $\chi$ and~$\chi'$, respectively. Thus $\rhobar_{E_1,p}^{ss} \simeq \rhobar_{E_2,p}^{ss}$ for all $E_1, E_2 \in S$.

\begin{lemma} Let $E/\Q$ and~$C/\Q$ be elliptic curves such that
$\rhobar_{E,5} \simeq \left(\begin{smallmatrix}
  \chi_E & * \\
  0 & \chi_E'
\end{smallmatrix}\right)$
and $\rhobar_{C,5} \simeq \left(\begin{smallmatrix}
  \chi_C & * \\
  0 & \chi_C'
\end{smallmatrix}\right)$
both with $* \neq 0$.
Suppose that $\chi_E$ and $\chi_C$ fix the field $F$ and that
$\chi_E'$ and $\chi_C'$ fix the field~$F' \neq F$.
Then $\chi_E = \chi_C$ and $\chi_E' = \chi_C'$.
\end{lemma}

\begin{proof} Since $\chi_E$ and~$\chi_E'$ are generators of the dual groups of $\Gal(F/\Q)$ and $\Gal(F'/\Q)$, respectively, we must have
$\chi_C = \chi_E^i$ and $\chi_C' = (\chi_E')^j$ with $i$ and~$j$ coprime to the order of~$\chi_E$ and~$\chi_E'$ respectively. The orders of $\chi_E$ and~$\chi_E'$ divide~$5-1=4$ so we can take $1 \leq i,j \leq 4$.

Since $\det \; \rhobar_{E,5}  =\det \; \rhobar_{C,5} = \chi_5$, we have
$\chi_E \chi_E' = \chi_5$ and $\chi_C \chi_C' = \chi_E^i (\chi_E')^j = \chi_5$.
In particular, if one of $\chi_E$,~$\chi_E'$ has order $< 4$ then the other must be of order 4; so, after replacing $E$ and~$C$ with their unique 5-isogenous curves if needed, we can assume that $\chi_E'$ is of order~4, and hence $j \in \{1,3\}$. We now divide into cases:

If $\chi_E = 1$ then $\chi_C = 1 = \chi_E$ and~$(\chi_E')^j = \chi_5^j = \chi_5$, hence $j=1$ and $\chi_E' = \chi_C'$, as desired.

If $\chi_E$ has order 2, then $i=1$ and $\chi_E \chi_E' (\chi_E')^{j-1} = \chi_5$, therefore $(\chi_E')^{j-1}=1$, so $j=1$ as desired.

If $\chi_E$ has order 4 then $i \in \{1,3 \}$. If $j=i$ then
$(\chi_E \chi_E')^j = \chi_5^j = \chi_5$ hence $j=1=i$ as desired. Suppose $i \neq j$. If $i=1$ and $j=3$, then
$\chi_5 (\chi_E')^2 = \chi_5$ and $\chi_E'$ has order dividing 2, a contradiction; if $i=3$ and $j=1$, then $\chi_E^2 \chi_5 = \chi_5$ so
$\chi_E$ has order dividing 2, and we are in one of the previous cases.
\end{proof}

7) From 6) we conclude that assumption all assumptions of \cite[Theorem~3.6]{CremonaFreitas} are satisfied for all $E_1, E_2 \in S$. Now, for each pair of non-isogenous curves $E_1, E_2 \in S(\chi)$, we compute the fields $F_1, F_2$ as in \cite[Theorem~3.6]{CremonaFreitas} and check whether they are isomorphic in which case we found a triple $(E_1,E_2,p)$; we do the same for $S(\chi')$.

\begin{remark}
For $p=5$ we could also use the alternative argument described in Remark~\ref{rem:mod5} instead of steps 6) and 7). Nevertheless, it is still important to work with curves with $*\neq 0$ as selected in 5), because otherwise we could find roots of the numerators
of~$j^\pm(t)-j(E')$ giving rise to $E'$ in both parametrizations, that is, the symplectic type of the isomorphisms $E[5] \simeq E'[5]$ is not uniquely defined.
\end{remark}

By the end of 7), we have determined all the reducible triples~$(E,E',p)$,  obtaining 22902 for $p=5$, 626 for $p=7$ and none for $p > 7$.

\subsection{Examples}
We finish with a few illustrative examples of application of our criteria.

From the lists of triples computed above, we selected the subset of triples satisfying the hypothesis of Theorem~\ref{thm:multiplicative} or Theorem~\ref{thm:mixedReduciton} at some prime~$\ell$.
We found zero triples for $p = 13, 17$, four triples for $p=11$, 1032 triples for $p=7$
and 31168 for $p=5$, all irreducible.

Moreover, we implemented in {\tt Magma}~\cite{magma} our criteria and applied them to determine the symplectic type of these triples; the symplectic types obtained are compatible with those previously computed using the methods from~\cite{CremonaFreitas,symplectic}, as expected.

All curves in the examples below are given by their Cremona labels.

\begin{example}
Up to quadratic twist, there are three irreducible congruences mod~$p=11$ satisfying our criteria. Namely, we have symplectic congruences between the curves
$$(E,E') = (2093b1, 26611d1), \; (497a1, 148603e1), \; (266526w1, 360594s1),$$
where the first pair of curves satisfies
Theorem~\ref{thm:mixedReduciton} at $\ell = 89$ and the last two pairs satisfy the same theorem at $\ell = 23$; in particular, there are no reducible triples satisfying either criteria, no triple satisfies Theorem~\ref{thm:multiplicative}, and no anti-symplectic congruences among the above.
\end{example}

\begin{example}
For $p=7$ no triples satisfy Theorem~\ref{thm:multiplicative}, hence all the 1032 satisfy Theorem~\ref{thm:mixedReduciton}, of which 743 are symplectic congruences and 289 anti-symplectic. Furthermore, there are 9 pairs of curves satisfying Theorem~\ref{thm:mixedReduciton} at two different values of~$\ell$, for example, the curves
$(85608s1,333576e1)$
satisfy Theorem~\ref{thm:mixedReduciton} at $\ell \in \{29,113\}$.
\end{example}

\begin{example}
For $p=5$ we found 150 triples satisfying Theorem~\ref{thm:multiplicative} among which 77 are symplectic and 73 anti-symplectic; we found also 31020 triples
satisfying Theorem~\ref{thm:mixedReduciton} of which 14936 are symplectic and 16084 anti-symplectic.
Note that there are 2 pairs that satisfy both theorems, for example, the pair $(66880bk1,414656bc1)$ satisfies Theorem~\ref{thm:multiplicative} at $\ell = 11$ and Theorem~\ref{thm:mixedReduciton} at $\ell = 31$. Moreover,
there are 592 pairs satisfy Theorem~\ref{thm:mixedReduciton}
at two different~$\ell$ and the pair $(291555a1,490770cw1)$ satisfies Theorem~\ref{thm:mixedReduciton} for all $\ell \in \{11,31,41\}$.
\end{example}

\begin{example} For $p=5$ there also 122 triples that satisfy Theorem~\ref{thm:multiplicative} up to unramified quadratic twist, and for $p \geq 7$ we found no such triple.

We give the details for $(E,E',5) = (227370t1, 227370v1,5)$.

Let $N = 2^4 \cdot 3 \cdot 5 \cdot 11 \cdot 13 \cdot 53$ and consider the curves
$E_1$ and~$E_2$ both of conductor~$N$
given by
\[
 E_1 \; : \; y^2 = x^3 + x^2 - 17930962498800x - 28929993729305827500
\]
and
\[
 E_2 \; : \; y^2 = x^3 + x^2 - 7090720x + 7507005428.
\]
They have minimal discriminants
\[
 \Delta_m(E_1) = 2^{13} 3^{13} 5^7 11^5 13^{13} 53^3 \quad \text{ and } \quad \Delta_m(E_2) = 2^{35} 3^4 5^1 11^5 13^1 53^1
\]
and $j$-invariants
\[
 j(E_1) = \frac{25463^3  \cdot 2112590327^3}{2 \cdot 3^{13} 5^7 11^5 13^{13} 53^3} \quad \text{ and } \quad j(E_2) = - \frac{2293^3 \cdot 9277^3}{2^{23} 3^4 5^1 11^5 13^1 53^1}.
\]
Since their conductor is $> 500 000$ these curves are currently not in the LMFDB, but can be obtained by taking quadratic twist by $-1$ of the curves $E$ and~$E'$, respectively. A consultation of the LMFDB shows that the mod~$5$ representations of~$E$ and~$E'$ are irreducible; moreover, from \cite[Proposition~4]{KO}, comparing traces of Frobenius at all primes
$q < 108864$  not dividing~$N$ shows that the $G_\Q$-modules $E[5]$ and $E'[5]$ are isomorphic, therefore the symplectic type of $(E_1,E_2,5)$ is well defined. We will now determine its symplectic type using Theorem~\ref{thm:multiplicative}. Indeed, both $E_1$ and $E_2$ have split multiplicative reduction at $\ell = 11$ and satisfy
\[
 v_{11}(j(E_1)) = v_{11}(j(E_2)) = -5 \equiv 0 \pmod{5}, \quad \tilde{j}_{E_1} \equiv 3 \pmod{11}
 \quad \text{ and } \quad \tilde{j}_{E_2} \equiv 5 \pmod{11}.
 \]
Moreover, $3$ is a $5th$ root of unity in $\F_{11}$, and computing logarithms
in $\mu_5 \subset \F_{11}^\times$ yields
\[
 -h_1 = \text{Log}_3(3) = 1 \quad \text{ and } \quad -h_2= \text{Log}_3(5) = 3,
\]
hence $h_1/ h_2 \equiv 2 \pmod{5}$ is not a square mod~$5$ thus the type of $(E_1,E_2,5)$ is anti-symplectic by Theorem~\ref{thm:multiplicative}. As a sanity check, we can apply the original multiplicative criterion at $\ell = 53$; indeed,
$v_{53}(\Delta_m(E_i)) \not\equiv 0 \pmod{5}$ and $v(\Delta_m(E_1)/v(\Delta_m(E_2)) \equiv 3 \pmod{5}$ is not a square mod~$5$, therefore the symplectic type of $(E_1,E_2,5)$ is anti-symplectic by \cite[Proposition 2]{KO}.
\end{example}

\begin{example} We give details of an application of Theorem~\ref{thm:mixedReduciton}.
Consider the curves with Cremona labels
 $E = 103540a1$ and $E' = 3340a1$. A quick consultation of the LMFDB shows that their mod~$5$ representations are irreducible and that
$a_{31}(f_{E})a_{31}(E') \equiv (31 + 1) \pmod{5}$ where $f_{E}$ is the modular form corresponding to~$E$ via modularity. Therefore, from \cite[Proposition~4]{KO}, comparing traces of Frobenius at all primes
$q < 32256$, $q \neq 31$ and ~$q \nmid 103540$ shows that $E[5] \simeq E'[5]$ as $G_\Q$-modules, and so the symplectic type of $(E,E',5)$ is well defined. We will now determine it using Theorem~\ref{thm:mixedReduciton}. Indeed, the curve $E$ has split multiplicative reduction at $31$; moreover, $2$ is a $5th$ root of unity in $\F_{31}$ and we have
\[
 v_{31}(j(E)) = -5 \equiv 0 \pmod{5}, \quad \tilde{j}_{E} \equiv 16 \pmod{31},
 \quad -h = \text{Log}_2(16) = 4,
\]
where we computed the logarithms
in $\mu_5 \subset \F_{31}^\times$.
The curve $E'$ has good reduction at~$\ell = 31$; to compute $h'$ we use the fact that the reduction morphism at primes of good reduction induces a symplectic Galois equivariant isomorphism. Indeed, $(E'/\Q_{31})[5] \simeq (E'/\F_{31})[5]$
where $E' / \F_{31} : y^2 = x^3 + 30x^2 + 30x + 6$ and an explicit computation of
the $5$-torsion of $E' / \F_{31}$  shows that $h' = 1$. Since $-h/h' \equiv 2^2 \pmod{5}$ we conclude that the type of $(E,E',5)$ is symplectic by Theorem~\ref{thm:mixedReduciton}.
As a sanity check, we apply the original multiplicative criterion at $\ell = 167$; indeed, both curves have multiplicative reduction at~$\ell$ and
$v_{167}(\Delta_m(E)) = v_{167}(\Delta_m(E')) = 1$; thus $v_{167}(\Delta_m(E))/v_{167}(\Delta_m(E')) = 1$ is a square mod~$5$ and the type of $(E,E',5)$ is symplectic by \cite[Proposition 2]{KO}.
\end{example}

\subsection{Reducible examples for $p=5$} The exhaustive search for mod~$5$ congruences between elliptic curves of conductor $< 500000$ discussed above did not yield any example of a reducible congruence whose symplectic type can be determined using our criteria.
For completeness, we will now give examples of application of our criteria to congruences between
elliptic curves with mod~$5$ representation of the form
$\left(\begin{smallmatrix}
  \chi_5 & * \\ 0 & 1
\end{smallmatrix} \right)$.

We start by presenting the general construction of such representations from elliptic curves.

Let  $K$ be a field of characteristic zero, and $t$ an element of $K$. We consider the affine cubic defined over $K$ by the equation
$$y^2+a_1xy+a_3 y=x^3+a_2 x^2+a_4 x+a_6,$$
with
$$a_1=1-t,\quad a_2=a_3=-t,\quad a_4=-5(t^3+2t^2-t),\quad a_6=-t(t^4+10t^3-5t^2+15t-1).$$
The standard invariants $(c_4,c_6,\Delta)$ associated to this Weierstrass equation are
$$c_4=t^4+228t^3+494t^2-228t+1,\quad c_6=-(t^2+1)(t^4-522t^3-10006t^2+522t+1),$$
$$\Delta=t(t^2-11t-1)^5.$$
From now on suppose  $\Delta\neq 0$, so that the cubic defines an elliptic curve $E(t)$ defined over $K$.  Let us recall some properties of the Galois module $E(t)[5]$; see \cite{Kraus} for details.
Let $\mu_5 \subset \Qbar$ be set of $5$-th roots of unity.
Fix $\zeta \in \mu_5$ and let $P_t=(X,Y)$ be the point defined by
$$X=-\frac{3t^2-8t+2}{5}-\frac{t^2-11t-1}{5} (\zeta^2+\zeta^3),$$
$$Y=-\frac{t(t^2-t+14)}{5} + \frac{t^3-9t^2-23t-2}{5} \zeta -\frac{t^3-13t^2+21t+2}{5} \zeta^2 +\frac{t^3-10t^2-12t-1}{5} \zeta^3.$$

This  point belongs to $E(t)$ and  is rational over $K(\mu_5)$.  Moreover, the map
$$i_t : \mu_5\to E(t)[5]$$
defined by $i_t(\zeta)=P_t$,  is an injective group homomorphism which is compatible with the Galois actions on $\mu_5$ and $E(t)[5]$. The field of the $5$-torsion points of $E(t)$
is
$$K(E(t)[5])=K(\mu_5,\root{5}\of{t}).$$
Furthermore,  for all $t$ and $t'$ in $K$, the Galois modules $E(t)[5]$ and $E(t')[5]$ are isomorphic if and only if $K(E(t)[5])=K(E(t')[5])$.
We can now produce examples.

\begin{example}
Let $t=2 \cdot 11^5$ and $t'=2/11^5$ and denote $E_1=E(t)$ and $E_2=E(t')$.
These curves are not part of the LMFDB as their conductors $N_{E_i}$  are
$$N_{E_1}=2 \cdot 11 \cdot 151\cdot 8831\cdot 77801\quad \text{and}\quad N_{E_2}=2\cdot 11\cdot 251\cdot 103350469.$$
From the above, the $\Gal(\overline \Q/\Q)$-modules $E_1[5]$ and $E_2[5]$ are isomorphic. In particular, by taking
$\ell=11$ they are isomorphic as $G_{\Q_{\ell}}$-modules. One easily checks that $E_1$ and $E_2$ have split multiplicative reduction at~$\ell$,
and for $i=1,2$ the conditions
$$v(j_{E_i})\equiv 0 \pmod 5\quad \text{and}\quad \tilde j_i^{\frac{\ell-1}{5}}\not\equiv 1 \pmod{11}$$
are satisfied.
In $\F_{11}^*$ one can take $\zeta=4$. One has
$$ \tilde j_1^{\frac{\ell-1}{5}}\equiv  3 \pmod{\ell}\quad \text{and}\quad \tilde j_2^{\frac{\ell-1}{5}}\equiv 4\pmod{\ell}.$$
Note that, from Theorem~\ref{thm:multiplicative}, we recover here the fact that the $G_{\Q_{\ell}}$-modules $E_1[5]$ and $E_2[5]$ are isomorphic. Consequently, we have $h_1=4$ and $h_2=1$.
Hence, $h_1/h_2$ is a square modulo~$5$ and we deduce from Theorem~\ref{thm:multiplicative} that  $E_1[5]$ and $E_2[5]$ are symplectically isomorphic.
\end{example}

\begin{example} Let $t''=2$ and $E_3=E(t'')$. This is the curve with Cremona label~38b2, hence it has good reduction at $\ell = 11$.
As above, the $\Gal(\overline \Q/\Q)$-modules $E_1[5]$, $E_2[5]$ and $E_3[5]$ are isomorphic. Using {\tt Magma}, we check that in a suitable symplectic basis of $E_3[5]$ we have $h_3=1$. Let $h_1$ and~$h_2$ be as in the previous example. In particular, $-h_1/h_3$ and $-h_2/h_3$ are squares modulo $5$ and from Theorems~\ref{thm:mixedReduciton} and the previous example, we conclude that the Galois modules $E_1[5]$, $E_2[5]$ and $E_3[5]$ are pairwise symplectically isomorphic.
\end{example}

\end{document}